\documentclass[10pt,leqno]{article} 

   \usepackage[centertags]{amsmath}
   \usepackage{amsfonts}
   \usepackage{amsmath}
   \usepackage{amssymb}
   \usepackage{amsthm}
   \usepackage{newlfont}
   \usepackage[latin1]{inputenc}
    \usepackage{multirow, bigdelim}

\allowdisplaybreaks[3]

\theoremstyle{plain}
\newtheorem{theorem}{Theorem}[section]
\newtheorem{lemma}[theorem]{Lemma}

\theoremstyle{definition}

\theoremstyle{remark}

\newtheorem*{remark*}{Remark}

\numberwithin{equation}{section}

\newcommand\Cc{{\mathtt{C}}}
\newcommand\Hh{{\mathtt{H}}}
\newcommand\Mm{{\mathtt{M}}}
\newcommand\Ll{{\mathtt{L}}}
\newcommand\Pp{{\mathtt{P}}}

\newcommand\D{{\mathcal D}}

\newcommand\CC{{\mathbb C}}

\newcommand\RR{{\mathbb R}}
\newcommand\ZZ{{\mathbb Z}}
\newcommand\NN{{\mathbb N}}
\newcommand\PP{{\mathbb P}}

\newcommand\q{\mathfrak{q}}

\newcommand\pp{\mbox{$\mathfrak{p}_{F}$}}

\newcommand\F{{\mathcal F}}

\newcommand\dgr{\operatorname{dgr}}

\newcommand\Sh{\mbox{\Large $\mathfrak {s}$}}

\newcommand{\dosfilas}[2]{
  \ldelim[{2}{2mm}& #1 &\rdelim]{2}{2mm} \\
  & #2 & &  & &
}

\newcommand*\pFqskip{8mu}
\catcode`,\active
\newcommand*\pFq{\begingroup
        \catcode`\,\active
        \def ,{\mskip\pFqskip\relax}%
        \dopFq
}
\catcode`\,12
\def\dopFq#1#2#3#4#5{%
        {}_{#1}F_{#2}\biggl(\genfrac..{0pt}{}{#3}{#4};#5\biggr)%
        \endgroup
}

   \title{Invariant properties for Wronskian type determinants of classical and classical discrete orthogonal polynomials
under an involution of sets of positive integers
  \footnote{Partially supported by MTM2015-65888-C4-1-P (Ministerio de Economía y Competitividad),
FQM-262, FQM-7276 (Junta de Andalucía) and Feder Funds (European
Union).}}
   \author{Guillermo P. Curbera and Antonio J. Dur\'{a}n\\
     \footnotesize
        \  IMUS \& Departamento de An\'{a}lisis Matem\'{a}tico.
       Universidad de Sevilla \\
       \footnotesize Apdo (P. O. BOX) 1160. 41080 Sevilla. Spain.
   curbera@us.es; duran@us.es \\
          \ \ }

   \date{}
   
\begin{document}
   \maketitle

\bigskip

\begin{abstract}
Given a finite set $F=\{f_1,\cdots ,f_k\}$ of nonnegative integers (written in increasing size) and a classical discrete family $(p_n)_n$ of orthogonal polynomials (Charlier, Meixner, Krawtchouk or Hahn), we consider the Casorati determinant $\det(p_{f_i}(x+j-1))_{i,j=1,\cdots,k}$. In this paper we prove a nice invariant property for this kind of Casorati determinants when the set $F$ is changed by $I(F)=\{0,1,2,\cdots,\max F\}\setminus \{\max F-f:f\in F\}$. This symmetry is related to the existence of higher order difference equations for the orthogonal polynomials with respect to certain Christoffel transforms of the classical discrete measures. By passing to the limit, this invariant property is extended for Wronskian type determinants whose entries are Hermite, Laguerre and Jacobi polynomials.
\end{abstract}

\section{Introduction}
Wronskian and Casoratian determinants whose entries are orthogonal polynomials belonging to the Askey and $\q$-Askey schemes satisfy some very impressive invariance properties (se \cite{du0,du1,duar,GUGM2,OS0,OS1}). They have been found using different approaches.
S. Odake and R. Sasaki in \cite{OS0,OS1} use the equivalence between eigenstate adding and deleting Darboux transformations for solvable (discrete) quantum mechanical systems. This is also the approach used by D. Gómez-Ullate, Y. Grandati and R. Milson in \cite{GUGM2}. On the other hand, the approach used by one of us in \cite{du1} is based in certain purely algebraic transformations of a Wronskian type determinant whose entries are orthogonal polynomials. These Wronskian type determinants are of the form
\begin{equation}\label{mtd}
\det \left( T^{i-1}(p_{m+j-1}(x))\right)_{i,j=1}^n,
\end{equation}
where $m\in \NN$, $(p_n)_n$ is a sequence of orthogonal polynomials with respect to a measure $\mu$ and $T$ is  a linear operator acting in the linear space of polynomials $\PP$ and satisfying that $\dgr(T(p))=\dgr(p)-1$, for all polynomials $p$. The case $T=d/dx$ was studied by Leclerc in \cite{Le2}. Leclerc clarified the approach used by Karlin and Szeg\H o in \cite{KS} to generalize
the classical Tur\'an inequality for Legendre polynomials \cite{Tu} to Hankel determinants whose entries are
some families of classical and classical discrete polynomials. Karlin and Szeg\H o's strategy was to express these Hankel determinants in terms of the Wronskian of certain orthogonal polynomials of another class (see, also, \cite{DD,DPV,FZ,Is,IL}).

The purpose of this paper is to introduce other approach to find these invariance properties. This approach is based in the so-called Krall discrete polynomials. A Krall discrete orthogonal family is a sequence of polynomials $(p_n)_{n\in \NN}$, $p_n$ of degree $n$, orthogonal with respect to a measure which, in addition, are also eigenfunctions of a higher order difference operator. Krall discrete polynomials are one of the most important extensions of the classical discrete families of Charlier, Meixner, Krawtchouk and Hahn. A huge amount of families of Krall discrete orthogonal polynomials have been recently introduced by one of us by mean of certain Christoffel transform of the classical discrete measures  (see \cite{dur0,dur1,dudh,DdI,DdI2}). A Christoffel transformation consists in multiplying a measure $d\mu$ by a polynomial $r$. It has a long tradition in the context of orthogonal polynomials: it goes back a century and a half ago when E.B. Christoffel (see \cite{Chr} and also \cite{Sz}) studied it for the particular case $\mu =x$. Consider now such a Krall discrete measure $\mu$. One can then explicitly find two determinantal representations for the orthogonal polynomials with respect to $\mu$ using two different approaches. On the one hand, one can easily find
a determinantal representation for them by using the general theory of Christoffel transform (see Section \ref{secChr} below). On the other hand, taking into account that the orthogonal polynomials with respect to $\mu$ are eigenfunctions of a higher order difference operator, one can find other determinantal representation for them by applying the method of $\D$-operators developed in \cite{dur1,DdI,DdI2,dudh}.
Most of the symmetries for Casoratian determinants  proved in \cite{du1} are underlying the fact that the orthogonal polynomials with respect to certain Krall discrete measures admit both determinantal representations. Actually, in this paper we prove invariance properties for a bunch of new variants of Casoratian determinants by exploiting this fact. Moreover, by passing to the limit we extend these invariance properties to Wronskian whose entries are classical polynomials. This approach establishes a connection with the technique used in \cite{OS0,OS1,GUGM}. At the discrete level, duality between the variable $x$ and the index $n$ of the polynomial (roughly speaking) is a well-known and fruitful concept. It has been shown in \cite{duch,dume,duhj} that duality interchanges exceptional discrete polynomials with Krall discrete polynomials; one can them construct exceptional polynomials by taking limit. Exceptional  polynomials allow to write exact solutions to rational extensions of classical quantum potentials (continuous and discrete). The last few years have seen a great deal of activity in the area  of exceptional polynomials, mainly by theoretical physicists (see, for instance,
\cite{duch,dume,duhj,GUKM1,GUKM2} (where the adjective \textrm{exceptional} for this topic was introduced), \cite{GUGM,G,OS-1,OS4,Qu}, and the references therein). The approach used by  Odake and Sasaki, and Gómez-Ullate, Grandati and Milson to study invariance properties for Wronskian and Casoratian determinants whose entries are  polynomials belonging to the Askey and $\q$-Askey schemes is based in these exceptional polynomials.

For the benefit of the reader, we display here in detail the invariance property for Casoratian determinants whose entries are Charlier polynomials which we prove in this paper.
To do that we need to introduce the following mapping $I$ defined in the set $\Upsilon_0$  formed by all finite sets of nonnegative integers:
\begin{align*}\noindent
I&:\Upsilon_0 \to \Upsilon_0 \\\label{dinv}
I(F)=\{0,1,2,&\cdots, \max F\}\setminus \{\max F-f,f\in F\}.
\end{align*}
It is not difficult to see that $I$ is an involution when restricted to the set $\Upsilon$  formed by all finite sets of positive integers.

For $a\not =0$, we write $(c_{n}^{a})_n$ for the sequence of Charlier polynomials normalized by taking its leading coefficient equal to $1/n!$ (see Section \ref{ich} below). They are orthogonal with respect to the measure
\begin{equation}\label{chwi}
\mu_a=\sum_{x=0}^\infty \frac{a^x}{x!}\delta_x.
\end{equation}
For a finite set $F=\{f_1,\cdots , f_k\}$ of nonnegative integers (written in increasing size), we  write $\Cc_{F,x}^{a}$ for the Casorati-Charlier determinant
\begin{equation}\label{cch}
\Cc_{F,x}^{a}= \big| c_{f_i}^{a}(x+j-1)\big)\big| _{i,j=1}^k,
\end{equation}
(we use the standard notation $\vert M\vert$ to denote the determinant of the matrix $M$).
It is not difficult to see that $\Cc_{F,x}^a$ is a polynomial in $x$ of degree $w_F=\sum_Ff-\binom{k}{2}$ (see \cite{duch}, Sect. 3).
We then prove the following invariance property for $\Cc_{F,x}^a$.

\begin{theorem}\label{tch}
For $a\not =0$ and a finite set $F$ of nonnegative integers, the invariance
\begin{equation}\label{cip2}
\Cc_{F,x}^{a}=(-1)^{w_F}\Cc_{I(F),-x}^{-a}
\end{equation}
holds.
\end{theorem}

Identity (\ref{cip2}) was conjectured by one of us in \cite{duch}. For segments $F=\{n,n+1,\cdots, n+k-1\}$, identity (\ref{cip2}) was conjectured in \cite{du0} and proved in \cite{du1}.

Theorem \ref{tch} will be proved in Section \ref{ich} using the two different determinantal representations mentioned above for the orthogonal polynomials with respect to the  measure
$$
\mu_a^F=\sum_{x=0}^\infty \prod_{f\in F}(x-f) \frac{a^x}{x!}\delta_x.
$$

Passing to the limit, the identity (\ref{cip2}) produces the following invariance for (normalized) Wronskian whose entries are Hermite polynomials (see Section \ref{ihe}). For a finite set $F$ of nonnegative integers, write
\begin{equation}\label{wh}
\Hh_{F,x}=\frac{1}{2^{\binom{k}{2}}\prod_{f\in F}f!}\vert H_{f_i}^{(j-1)}(x)\vert_{i,j=1}^k,
\end{equation}
where as usual $H_n$ denotes de $n$-th Hermite polynomial.

\begin{theorem}\label{the} For a finite set  $F$ of nonnegative integers the invariance
\begin{equation}\label{hip}
\Hh_{F,x}=i^{w_F}\Hh_{I(F),-ix}
\end{equation}
holds.
\end{theorem}

Identity (\ref{hip}) was also conjectured in \cite{duch}.

The case of Meixner, Hahn, Laguerre and Jacobi is even richer than this of Charlier and Hermite. We study them in Sections \ref{ime}, \ref{sha}, \ref{sla} and \ref{sja}, respectively.

\section{Preliminaries}
Let $\mu $ be a Borel measure (positive or not) on the real line. The $n$-th moment of $\mu $ is defined by
$\int _\RR t^nd\mu (t)$. When $\mu$ has finite moments for any $n\in \NN$, we can associate it a bilinear form defined in the linear space of polynomials by
\begin{equation}\label{bf}
\langle p, q\rangle =\int pqd\mu.
\end{equation}
We say that the polynomials $p_n$, $n\in \NN$, $p_n$ of degree $n$, are orthogonal with respect to $\mu$ if they
are orthogonal with respect to the bilinear form defined by $\mu$; that is, if they satisfy
$$
\int p_np_md\mu =c_n\delta_{n,m}, \quad c_n\not =0, \quad n \in \NN.
$$
Orthogonal polynomials with respect to a measure are unique up to multiplication by non null constant.
Positive measures $\mu $ with finite moments of any order and infinitely many points in its support has always a sequence of orthogonal polynomials $(p_n)_{n\in\NN }$, $p_n$ of degree $n$ (it is enough to apply the Gram-Smith orthogonalizing process to $1, x, x^2, \ldots$); in this case
the orthogonal polynomials have positive norm: $\langle p_n,p_n\rangle>0$. Moreover, given a sequence of orthogonal polynomials $(p_n)_{n\in \NN}$ with respect to a measure $\mu$ (positive or not) the bilinear form (\ref{bf}) can be represented by a positive measure if and only if $\langle p_n,p_n \rangle > 0$, $n\ge 0$.

As usual $(a)_n$, $a\in \ZZ, n\in \NN$ denotes the Pochhammer symbol defined by $(a)_n=a(a+1)\cdots(a+n-1)$.

\subsection{Christoffel transform}\label{secChr}
Let $\mu$ be a measure (positive or not) and assume that $\mu$ has a sequence of orthogonal polynomials
$(p_n)_{n\in \NN}$, $p_n$ with degree $n$ and $\langle p_n,p_n\rangle \not =0$ (as we mentioned above, that always happens if $\mu$ is positive, with finite moments and infinitely many points in its support).

Given a finite set $F$ of real numbers, $F=\{f_1,\cdots , f_k\}$, $f_i<f_{i+1}$, we write $\Phi_n$, $n\ge 0$, for the $k\times k$ determinant
\begin{equation}\label{defph}
\Phi_n=\vert p_{n+j-1}(f_i)\vert _{i,j=1,\cdots , k}.
\end{equation}
Notice that $\Phi_n$, $n\ge 0$, depends on both, the finite set $F$ and the measure $\mu$.

The Christoffel transform of $\mu$ associated to the annihilator polynomial $\pp$ of $F$,
$$
\pp (x)=(x-f_1)\cdots (x-f_k),
$$
is the measure defined by $ \mu_F =\pp \mu$.

Orthogonal polynomials with respect to $\mu_F$ can be constructed by means of the formula
\begin{equation}\label{mata00}
q_n(x)=\frac{1}{\pp (x)} \begin{vmatrix}p_n(x)&p_{n+1}(x)&\cdots &p_{n+k}(x)\\
p_n(f_1)&p_{n+1}(f_1)&\cdots &p_{n+k}(f_1)\\
\vdots&\vdots&\ddots &\vdots\\
p_n(f_k)&p_{n+1}(f_k)&\cdots &p_{n+k}(f_k) \end{vmatrix}.
\end{equation}
Notice that the degree of $q_n$ is equal to $n$ if and only if $\Phi_n\not =0$. In that case the leading coefficient $\lambda^Q_n$ of $q_n$ is
equal to $(-1)^k\lambda^P_{n+k}\Phi_n$, where $\lambda ^P_n$ denotes the leading coefficient of $p_n$.

The next Lemma follows easily using \cite{Sz}, Th. 2.5.

\begin{lemma}\label{sze}
The measure $\mu_F$ has a sequence $(q_n)_{n=0}^\infty $, $q_n$ of degree $n$, of orthogonal polynomials if and only if $\Phi_n\not =0$, $n\in \NN$.
In that case, an orthogonal polynomial of degree $n$ with respect to $\mu _F$ is given by (\ref{mata00}) and also $\langle q_n,q_n\rangle _{\mu _F}\not =0$, $n\ge 0$. Moreover
\begin{equation}\label{n2q}
\langle q_n,q_n\rangle _{\mu_F}=(-1)^k\frac{\lambda^P_{n+k}}{\lambda^P_{n}}\Phi_n\Phi_{n+1}\langle p_n,p_n\rangle _{\mu}.
\end{equation}
\end{lemma}

\subsection{Finite set of positive integers}\label{sfspi}
To each finite set $F$ of nonnegative integers $F=\{ f_1,\cdots , f_k\}$, with $f_i<f_{i+1}$, we associate
the nonnegative integer $w_F$  defined by
\begin{equation}\label{defuf}
w_F=\sum_{f\in F}f-\binom{k}{2}.
\end{equation}

Consider the set $\Upsilon_0$  formed by all finite sets of nonnegative  integers and denote by $\Upsilon$ the subset of $\Upsilon_0$ formed by all finite sets of positive integers.

We consider the mapping $I$ in $\Upsilon_0$ defined by
\begin{align}\label{dinv}
I(F)=\{0,1,2,\cdots, \max F\}\setminus \{\max F-f,f\in F\}.
\end{align}
When restricted to $\Upsilon$, I is an involution: $I^2=Id$.

For the involution $I$, the bigger the holes in $F$ (with respect to the set $\{0,1,2,\cdots , f_k\}$), the bigger the involuted set $I(F)$.
Here it is a couple of examples
$$
I(\{ 1,2,3,\cdots ,k\})=\{ k\},\quad \quad I(\{1, k\})=\{ 1,2,\cdots, k-2, k\}.
$$
The set $I(F)$ will be denoted by $G$: $G=I(F)$. We also write $G=\{g_1,\cdots , g_m\}$ with $g_i<g_{i+1}$ so that $m$ is the number of elements of $G$ and $g_m$ the maximum element of $G$. Notice that the number $w_F$ (\ref{defuf}) is invariant in $\Upsilon$ under $I$. Moreover
\begin{equation}\label{kym}
\mbox{if $F\in \Upsilon$ then } \begin{cases}w_F=w_{I(F)},&\\ f_k=g_m,&\\ m=f_k-k+1.&\end{cases}
\end{equation}
For $F\in \Upsilon$, we also define the number $s_F$ by
\begin{equation}\label{defs0}
s_F=\begin{cases} 1,& \mbox{if $F=\emptyset$},\\
k+1,&\mbox{if $F=\{1,2,\cdots , k\}$},\\
\min \{s\ge 1:s<f_s\}, & \mbox{if $F\not =\{1,2,\cdots k\}$}.
\end{cases}
\end{equation}
For $F\in \Upsilon_0$, we denote by  $F_{\Downarrow}$ the finite set of positive integers defined by
\begin{equation}\label{deff1}
F_{\Downarrow}=\begin{cases} \emptyset,& \mbox{if $F=\{1,2,\cdots , k\}$,}\\
\{f_{s_F}-s_F,\cdots , f_k-s_F\},& \mbox{if $F\not =\{1,2,\cdots , k\}$ and $0\not \in F$}\\
(F \setminus \{0\})_{\Downarrow},& \mbox{if $0\in F$}.
\end{cases}
\end{equation}
One straightforwardly has that
\begin{equation}\label{ref0f1}
\mbox{for $0\in F$ then} \begin{cases}I(F)=I(F_{\Downarrow}),&\\w_F=w_{F_{\Downarrow}}.&\end{cases}
\end{equation}
For an $n$-tuple $\F=(F_1,\cdots, F_n)$ of finite sets of nonnegative numbers, we define
\begin{align}\label{defufn}
w_\F&=\sum_{i=1}^nw_{F_i},\\\label{dinvn}
I(\F)&=(I(F_1),\cdots ,I(F_n)).
\end{align}

\section{Invariance for Casorati-Charlier determinants}\label{ich}
We start with some basic definitions and facts about Charlier  polynomials which we will need later.

For $a\neq0$, we write $(c_n^a)_n$ for the sequence of Charlier polynomials (the next formulas can be found in \cite{Ch}, pp. 170-1; see also \cite{KLS}, pp., 247-9 or \cite{NSU}, ch. 2) defined by
\begin{equation}\label{Chpol}
    c_n^a(x)=\frac{1}{n!}\sum_{j=0}^n(-a)^{n-j}\binom{n}{j}\binom{x}{j}j!.
\end{equation}
The Charlier polynomials are orthogonal with respect to the measure
(\ref{chwi}) (which is positive only when $a>0$) and satisfy
\begin{equation}\label{norCh}
\langle c_n^a,c_n^a\rangle=\frac{a^n}{n!}e^a.
\end{equation}
The three-term recurrence formula for $(c_n^a)_n$ is ($c_{-1}^a=0$)
\begin{equation}\label{Chttrr}
   xc_n^a=(n+1)c_{n+1}^a+(n+a)c_n^a+ac_{n-1}^a,\quad n\geq0.
\end{equation}
They are eigenfunctions of the following second-order difference operator
\begin{equation}\label{Chdeq}
   D_a=-x\Sh_{-1}+(x+a)\Sh_0-a\Sh_1,\quad D_a(c_n^a)=nc_n^a,\quad n\geq0,
\end{equation}
where $\Sh_j(f)=f(x+j)$. They also satisfy
\begin{equation}\label{Chlad}
   \Delta(c_n^a)=c_{n-1}^a,
\end{equation}
and the duality
\begin{equation}\label{Chdua}
   (-a)^mn!c_n^a(m)=(-a)^nm!c_{m}^a(n), \quad n,m\ge 0.
\end{equation}

It is not difficult to prove that for $a\not=0$ the Casoratian determinant $\Cc_{F,x}^a$ (\ref{cch}) is a polynomial in $x$ of degree $w_F$ (see Section 3 of \cite{duch}).

As explained in the Introduction, our starting point to prove Theorem \ref{tch} is the Krall-Charlier measure
\begin{equation}\label{kchm}
\mu_a^F=\sum_{x=0}^\infty \prod_{f\in F}(x-f) \frac{a^x}{x!}\delta_x,
\end{equation}
where $F=\{ f_1,\cdots, f_k\}$ is a finite set of positive integers (written in increasing size). Consider the $k\times k$ determinant
\begin{equation}\label{dph}
\Phi_n=\vert c_{n+j-1}^a(f_i)\vert_{i,j=1}^k.
\end{equation}
Assuming that $\Phi_n^{a;F}\not =0$, $n\in \NN$,
and using Lemma \ref{sze} for the Christoffel transform of a measure, one can generate a sequence of orthogonal polynomials with respect to $\mu_a^F$ by means of the determinant
\begin{equation}\label{defqnch}
q_n^{a;F}(x)=\frac{\begin{vmatrix}c_n^a(x)&c_{n+1}^a(x)&\cdots &c_{n+k}^a(x)\\
c_n^a(f_1)&c_{n+1}^a(f_1)&\cdots &c_{n+k}^a(f_1)\\
\vdots&\vdots&\ddots &\vdots\\
c_n^a(f_k)&c_{n+1}^a(f_k)&\cdots &c_{n+k}^a(f_k) \end{vmatrix}}{\prod_{f\in F}(x-f)} .
\end{equation}
But there still is other determinantal representation for the orthogonal polynomials with respect to $\mu_a^F$. Indeed, for the involuted set $G=I(F)=\{g_1,\cdots ,g_m\}$ (\ref{dinv}) (written also in increasing size), we consider the polynomials
\begin{equation}\label{quschi}
\tilde q_n^{a;F}(x)=\begin{vmatrix}
c^a_n(x-f_k-1) & -c^a_{n-1}(x-f_k-1) & \cdots & (-1)^mc^a_{n-m}(x-f_k-1) \\
c^{-a}_{g_1}(-n-1) & c^{-a}_{g_1}(-n) & \cdots &
c^{-a}_{g_1}(-n+m-1) \\
               \vdots & \vdots & \ddots & \vdots \\
               c^{-a}_{g_m}(-n-1) & \displaystyle
               c^{-a}_{g_m}(-n) & \cdots &c^{-a}_{g_m}(-n+m-1)
             \end{vmatrix}.
\end{equation}
Assuming that
$$
\Cc_{G,-n}^{-a}=\vert c_{g_i}^{-a}(-n+j-1)\vert_{i,j=1}^m\not =0,\quad n\in \NN,
$$
and using the $\D$-operators technique, it was prove in \cite{DdI} (see Theorem 1.1) that the polynomials $(\tilde q_n^{a;F}(x))_n$ are also orthogonal with respect to the Krall-Charlier measure (\ref{kchm}). Both determinantal representations (\ref{defqnch}) and (\ref{quschi})
are the key to prove Theorem \ref{tch}.

\begin{proof}[Proof of Theorem \ref{tch}]

We first prove that it is enough to consider finite sets $F$ of positive integers.

Indeed, combining columns in the Casorati-Charlier determinant $\Cc_{F,x}^a$ (\ref{cch}) and using (\ref{Chlad}), we can rewrite it as follows
\begin{equation}\label{alt}
\Cc_{F,x}^a=\vert c_{f_i-j+1}(x)\vert _{i,ju=1}^k.
\end{equation}
Assume now we have already proved Theorem \ref{tch} for finite sets  of positive integers. Let $F$ be a finite set of nonnegative integers with
$0\in F$. It is easy to prove using (\ref{alt}) that
$\Cc_{F,x}^a=\Cc_{F_\Downarrow,x}^a$, where the set $F_\Downarrow$ is defined by (\ref{deff1}). Since $0\not \in F_\Downarrow$, we have
$\Cc_{F_\Downarrow,x}^a=(-1)^{w_{F_\Downarrow}}\Cc_{I(F_\Downarrow),-x}^{-a}$. Since for $0\in F$, $I(F)=I(F_\Downarrow)$ and
$w_{F_\Downarrow} =w_F$ (see (\ref{ref0f1})), we then have
$\Cc_{F,x}^a=(-1)^{w_{F}}\Cc_{I(F),-x}^{-a}$. That is, Theorem \ref{tch} is also true when $0\in F$.

Hence, we can  assume that $F$ is a finite set of positive integers.

Since both $\Cc_{F,x}^{a}$ and $\Cc_{I(F),-x}^{-a}$ are polynomials in $x$, it is enough to prove Theorem \ref{tch} for $x=n\in \NN$.

Fixed $F$, we prove Theorem \ref{tch} for those $a$ satisfying the assumptions $\Phi_n^{a;F}\not =0$ (see (\ref{dph})) and $\Cc_{I(F),-n}^{-a}\not =0$, $n\in \NN$. By standard analyticity arguments, the result will then follow for all $a\not = 0$.

Since $(q_n^{a;F})_n$ (\ref{defqnch}) and $(\tilde q_n^{a;F})_n$ (\ref{quschi}) are orthogonal polynomials with respect to the same measure $\mu_a^F$ (\ref{kchm}) and
orthogonal polynomials with respect to a measure are unique up to multiplicative constants, we have that
\begin{equation}\label{ipo}
q_n^{a;F}(x)=\gamma_n\tilde q_n^{a;F}(x).
\end{equation}
On the one hand, comparing leading coefficients in (\ref{ipo}), we get (using Lemma \ref{sze}, the normalization for the Charlier polynomials and the definition of $\tilde q_n^{a;F}$)
$$
\frac{(-1)^k}{(n+k)!}\Phi_n^{a;F}=\frac{\gamma_n}{n!}\Cc_{I(F),-n}^{-a}.
$$
Using the duality (\ref{Chdua}), we have (see also \cite{duch}, identity (3.14))
\begin{equation}\label{duomph}
\Cc_{F,n}^a=\frac{\prod_{i=0}^{k-1}(n+i)!}{(-a)^{kn-w_F}\prod_{f\in F}f!}\Phi_n^{a;F},
\end{equation}
where $w_F$ is defined by (\ref{defuf}).
This gives for $\gamma_n$ the expression
\begin{equation}\label{ga1}
\gamma_n=\frac{(-1)^k(-a)^{kn-w_F}\prod_{f\in F}f!}{\prod_{i=1}^{k}(n+i)!}\frac{\Cc_{F,n}^a}{\Cc_{I(F),-n}^{-a}}.
\end{equation}
On the other hand, the identity (\ref{ipo}) gives
\begin{equation}\label{no1}
\Vert q_n^{a;F}\Vert _2^2=\gamma_n^2 \Vert \tilde q_n^{a;F}\Vert _2^2.
\end{equation}
The $L^2$-norm $\Vert q_n^{a;F}\Vert _2$ can be computed using Lemma \ref{sze} and (\ref{norCh})
$$
\Vert q_n^{a;F}\Vert _2^2=\frac{(-1)^ka^ne^a}{(n+k)!}\Phi_n^{a;F}\Phi_{n+1}^{a;F}.
$$
Inserting it in (\ref{no1}) and using the duality (\ref{duomph}) and (\ref{ga1}), we get after straightforward computations
\begin{equation}\label{no2}
\Vert \tilde q_n^{a;F}\Vert _2^2=\frac{a^{n+k}e^a}{n!}(\Cc_{I(F),-n}^{-a})^2\frac{\Cc_{F,n+1}^a}{\Cc_{F,n}^a}.
\end{equation}
We now compute the $L^2$-norm $\Vert \tilde q_n^{a;F}\Vert _2$ using a different approach.
Indeed, we have
$$
\Vert \tilde q_n^{a;F}\Vert _2^2=\frac{\Cc_{I(F),-n}^{-a}}{n!}\langle x^n,\tilde q_n^{a;F}(x)\rangle_{\mu_a^F}.
$$
According to Lemma 4.2 of \cite{DdI} (see the last formula in the proof of that Lemma in \cite{DdI}, p. 66), this gives
\begin{equation}\label{no3}
\Vert \tilde q_n^{a;F}\Vert _2^2=\frac{(-1)^md}{n!}\Cc_{I(F),-n}^{-a}\Cc_{I(F),-n-1}^{-a},
\end{equation}
where the number $d=d(n,a,F)$ is given by
\begin{align}\label{pcl}
d=(-1)^m&\Big( \langle (x-\max F-1)^n,c_{n-m}^a(x-\max F-1)\rangle_{\mu_a^F} \\\nonumber &\left.\hspace{.5cm}-(-1)^{n-1}e^a a^{g_m}\sum_{i=1}^m\frac{(-g_i-1)^nc_{g_i}^{-a}(-n+m-1)}{p'(g_i)c_{g_i}^{-a}(0)}\right),
\end{align}
$p(x)=\prod_{i=1}^m(x-g_i-1)$ and, as before, we write $I(F)=G=\{g_1,\cdots, g_m\}$.

We now claim that
\begin{equation}\label{cl1}
d=(-1)^ma^{n+k}e^a.
\end{equation}

Inserting it in (\ref{no3}) and using then (\ref{no2}), we get
\begin{equation}\label{inx}
\frac{\Cc_{F,n+1}^a}{\Cc_{F,n}^a}=\frac{\Cc_{I(F),-n-1}^{-a}}{\Cc_{I(F),-n}^{-a}}.
\end{equation}
Using (\ref{alt}) and taking into account that $c_n^a(0)=(-a)^n/n!$, a direct computation gives
$$
\Cc_{F,0}^a=\frac{(-a)^{w_F}V_F}{\prod _Ff!},
$$
where $V_F$ is the Vandermonde determinant defined by
\begin{equation}\label{defvdm}
V_F=\prod_{1=i<j=k}(f_j-f_i).
\end{equation}
A careful computation using the definition of $I(F)$ (\ref{dinv}) shows that
$V_F/\prod _Ff!=V_{I(F)}/\prod _{I(F)}f!$. Hence, since $w_F=w_{I(F)}$ (\ref{kym}), we get $\Cc_{F,0}^a=(-1)^{w_F}\Cc_{I(F),0}^{-a}$. The invariance (\ref{cip2}) can now be proved easily from (\ref{inx}) by induction on $n$.

We finally prove the Claim (\ref{cl1}).

To do that, we use the identity (4.13) of \cite{DdI}:
\begin{align}\label{4.13}
\langle(x-&\max F-1)^j,c_{n-l}^a(x-\max F -1)\rangle_{\mu_a^F}\\\nonumber &=(-1)^{n+l+m-1}e^a a^{g_m}\sum_{i=1}^m\frac{(-g_i-1)^jc_{g_i}^{-a}(-n+l-1)}{p'(g_i)c_{g_i}^{-a}(0)},\\
\end{align}
where $l=0,\cdots , m$ and $0\le j\le n-1$ (notice that the left hand side of (\ref{4.13}) is $0$ when $n-l<0$; this case is labeled (4.15) in \cite{DdI}).

We then proof (\ref{cl1}) by induction on $n$. For $n=0$, we have using (\ref{pcl}), the duality (\ref{Chdua}) and
taking into account that $c_n^a(0)=(-a)^n/n!$
\begin{align*}
d&=(-1)^m e^a a^{g_m}\sum_{i=1}^m\frac{c_{g_i}^{-a}(m-1)}{p'(g_i)c_{g_i}^{-a}(0)}\\
&=(-1)^m e^a a^{g_m}\sum_{i=1}^m\frac{a^{g_i-m+1}(m-1)!g_i!c_{m-1}^{-a}(g_i)}{p'(g_i)g_i!a^g_i}\\
&=(-1)^m e^a a^{g_m-m+1}(m-1)!\sum_{i=1}^m\frac{c_{m-1}^{-a}(g_i)}{p'(g_i)}.
\end{align*}
On the one hand $g_m-m+1=k$ (see (\ref{kym})) and, on the other hand, since $c_{m-1}^{-a}(x)$ is a polynomial of degree $m-1$ and the polynomial $p$ has degree $m$, we have from Lemma 2.1 of \cite{DdI}
$$
d=(-1)^m e^a a^{k}(m-1)!\frac{1}{(m-1)!}=(-1)^m e^a a^{k}.
$$
This is just (\ref{cl1}) for $n=0$.

We now prove (\ref{cl1}) for $n+1$. Using the three term recurrence formula (\ref{Chttrr}), we can write
\begin{align*}
\langle (x-f_k-1)^{n+1}&,c_{n+1-m}^a(x-f_k-1)\rangle_{\mu_a^F}\\=&
\langle (n-m+2)(x-f_k-1)^{n},c_{n+2-m}^a(x-f_k-1)\rangle_{\mu_a^F}\\&+\langle (n+1-m+a)(x-f_k-1)^{n},c_{n+1-m}^a(x-f_k-1)\rangle_{\mu_a^F}
\\&+\langle a(x-f_k-1)^{n},c_{n-m}^a(x-f_k-1)\rangle_{\mu_a^F},
\end{align*}
where we have set $\max F=f_k$.
Using (\ref{4.13}) for $n+2$, $j=n$, $l=m$ and $n+1$, $j=n$, $l=m$, respectively, and the induction hypothesis, we have after straightforward computations
\begin{align*}
\langle(x-f_k-1&)^{n+1},c_{n+1-m}^a(x-f_k-1)\rangle_{\mu_a^F}=a^{n+k+1}e^a\\
&+(-1)^{n-1}e^aa^{g_m}\sum_{i=1}^m\frac{(-g_i-1)^n}{p'(g_i)c_{g_i}^{-a}(0)}
\left[(n-m+2)c_{g_i}^{-a}(-n+m-3)\right. \\&\left.-(n+1-m+a)c_{g_i}^{-a}(-n+m-2)+ac_{g_i}^{-a}(-n+m-1)\right].
\end{align*}
Using the second order difference equation (\ref{Chdeq}) (changing $a$ to $-a$, $n$ to $g_i$ and $x$ to $-n+m-2$), we finally have
\begin{align*}
\langle(x-f_k-1&)^{n+1},c_{n+1-m}^a(x-f_k-1)\rangle_{\mu_a^F}\\&=a^{n+k+1}e^a
+(-1)^{n}e^aa^{g_m}\sum_{i=1}^m\frac{(-g_i-1)^{n+1}c_{g_i}^{-a}(-n+m-2)}{p'(g_i)c_{g_i}^{-a}(0)}.
\end{align*}
This is just (\ref{cl1}) for $n+1$.

\end{proof}

\section{Invariance for Wroskian whose entries are Hermite polynomials}\label{ihe}
We write $(H_n)_n$ for the sequence of Hermite polynomials defined by (see \cite{Ch}, Ch. V; see also \cite{KLS}, pp, 250-3)
\begin{equation}\label{Hpol}
 H_n(x)=n!\sum_{j=0}^{[n/2]}\frac{(-1)^j(2x)^{n-2j}}{j!(n-2j)!}.
\end{equation}
The Hermite polynomials are orthogonal with respect to the weight function $e^{-x^2}$, $x\in \RR$.
They satisfy $H_n'(x)=2nH_{n-1}(x)$.

One can obtain Hermite polynomials from Charlier polynomials using the limit
\begin{equation}\label{blchh}
\lim_{a\to \infty}\left(\frac{2}{a}\right)^{n/2}c_n^a(\sqrt {2a}x+a)=\frac{1}{n!}H_n(x)
\end{equation}
see \cite{KLS}, p. 249 (take into account that we are using a different normalization for Charlier polynomials to that in \cite{KLS}). The previous limit is uniform in compact sets of $\CC$.

Using the limit (\ref{blchh}) and the identity (\ref{alt}), one can get the Wronskian $\Hh_{F,x}$ (\ref{wh}) from the Casoratian determinant (\ref{cch}) (see \cite{duch}, Section 5). More precisely
\begin{equation}\label{lim1}
\lim_{a\to \infty}\left(\frac{2}{a}\right)^{(u_F+k)/2}\Cc_{F,\sqrt {2a}x+a}^a=\Hh_{F,x}.
\end{equation}
Theorem \ref{the} is then an easy consequence of Theorem \ref{tch} and (\ref{lim1}).

\section{Invariance for quasi Casorati-Meixner determinants}\label{ime}
For $a\not =0, 1$ we write $(m_{n}^{a,c})_n$ for the sequence of Meixner polynomials defined by
\begin{equation}\label{Mxpol}
m_{n}^{a,c}(x)=\frac{a^n}{(1-a)^n}\sum _{j=0}^n a^{-j}\binom{x}{j}\binom{-x-c}{n-j}
\end{equation}
(we have taken a slightly different normalization from the one used in \cite{Ch}, pp. 175-7; see also \cite{KLS}, pp, 234-7 or \cite{NSU}, ch. 2).

For $a\not =0,1$ and $c\not =0,-1,-2,\ldots $, Meixner polynomials are always orthogonal with respect to measure $\rho_{a,c}$. For $0<\vert a\vert<1$ and $c\not =0,-1,-2,\ldots $, we have
\begin{equation*}\label{MXw}
\mu_{a,c}=\sum _{x=0}^\infty \frac{a^{x}\Gamma(x+c)}{x!}\delta _x.
\end{equation*}
Meixner polynomials satisfy the duality
\begin{equation}\label{duame}
a^{m-n}n!(1+c)_{m-1}m_{n}^{a,c}(m)=(a-1)^{m-n}m!(1+c)_{n-1}m_{m}^{a,c}(n).
\end{equation}

Krall-Meixner measures seem to have a richer structure than the Krall-Charlier measures considered in Section \ref{ich}. Indeed,
given a pair $\F=(F_1,F_2)$ of finite sets of nonnegative integers, $F_i$ with $k_i$ elements, $i=1,2$, respectively, one can construct
Krall-Meixner measures by multiplying the Meixner weight by the polynomial $\prod_{f\in F_1}(x-f)\prod_{f\in F_2}(x+c+f)$ (see \cite{dur0,dur1,DdI}). This produces the measure
$$
\mu _{a,c}^{\F}=\prod_{f\in F_1}(x-f)\prod_{f\in F_2}(x+c+f)\mu _{a,c}.
$$
One then can generate orthogonal polynomials with respect to the measure $\mu _{a,c}^{\F}$ using Lemma \ref{sze} ($k=k_1+k_2$):
\begin{equation}\label{defqnme}
q_n^{a,c;\F}(x)=\frac{\left|
  \begin{array}{@{}c@{}lccc@{}c@{}}
  & m_{n+j-1}^{a,c}(x) &&\hspace{-.9cm}{}_{1\le j\le k+1} \\
    \dosfilas{ m_{n+j-1}^{a,c}(f) }{f\in F_1} \\
    \dosfilas{(-1)^{j-1}m_{n+j-1}^{1/a,c}(f) }{f\in F_2}
  \end{array}
  \right|}{(-1)^{nk_2}\prod_{f\in F_1}(x-f)\prod_{f\in F_2}(x+c+f)},
\end{equation}
where we have used that $m_n^{a,c}(x)=(-1)^nm_n^{1/a,c}(-x-c)$.

Along the rest of this paper, we use the following notation:
given a finite set of positive integers $F=\{f_1,\ldots , f_k\}$, the expression
\begin{equation}\label{defdosf}
  \begin{array}{@{}c@{}lccc@{}c@{}}
  &  &&\hspace{-.9cm}{}_{1\le j\le k} \\
    \dosfilas{ z_{f,j}  }{f\in F}
  \end{array}
\end{equation}
inside of a matrix or a determinant will mean the submatrix defined by
$$
\left(
\begin{array}{cccc}
z_{f_1,1} & z_{f_1,2} &\cdots  & z_{f_1,k}\\
\vdots &\vdots &\ddots &\vdots \\
z_{f_k,1} & z_{f_k,2} &\cdots  & z_{f_k,k}
\end{array}
\right) .
$$
The determinant (\ref{defmex}) should be understood in this form.

In view of (\ref{defqnme}), and using the duality (\ref{duame}), we define the (normalized) quasi Casoratian-Meixner determinant
\begin{equation}\label{defmex}
\Mm_{\F,x}^{a,c}=\frac{\left|
  \begin{array}{@{}c@{}lccc@{}c@{}}
     &  &&\hspace{-.9cm}{}_{1\le j\le k} \\
     \dosfilas{ m_{f}^{a,c}(x+j-1) }{f\in F_1} \\
    \dosfilas{ m_{f}^{1/a,c}(x+j-1)/a^{j-1} }{f\in F_2}
  \end{array}
  \right|}{a^{\binom{k_2}{2}-k_2(k-1)}(1-a)^{k_1k_2}}.
\end{equation}
It can be proved that for $a\not=0,1$ this Casoratian determinant $\Mm_{\F,x}^a$  is a polynomial in $x$ of degree $w_\F$ (\ref{defufn}) (see Section 3 of \cite{dume}).

Using that $\mu_{a,c}^\F$ is a Krall discrete measure (its orthogonal polynomials are also eigenfunctions of a higher order difference operator), in  Theorem 1.1 of \cite{DdI} it is proved that the polynomials
\begin{equation}\label{qusmei}
\tilde q_n^{a,c;\F}(x)=\left|
  \begin{array}{@{}c@{}lccc@{}c@{}}
    &m_{n-j+1}^{a,\tilde c}(x-\max F_1-1)&& \hspace{-.9cm}{}_{1\le j\le m+1} \\
    \dosfilas{ m_{g}^{a,2-\tilde c}(-n+j-2) }{g\in I(F_1)} \\
    \dosfilas{ m_{g}^{1/a,2-\tilde c}(-n+j-2)/a^{j-1} }{g\in I(F_2)}
  \end{array}
  \right|,
\end{equation}
where $\tilde c=c+\max F_1+\max F_2+2$ and $m$ is the sum of the number of elements of $I(F_1)$ and $I(F_2)$, are also orthogonal with respect to $\mu_{a,c}^\F$.

We then have $q_n^{a,c;\F}(x)=\gamma_n\tilde q_n^{a,c;\F}(x)$ for certain normalization sequence $\gamma_n\not =0$, $n\in \NN$.
One can then compute explicitly the sequence $\gamma_n$ by comparing the leading coefficient and the $L^2$-norm of the polynomials $q_n^{a,c;\F}(x)$ and $\tilde q_n^{a,c;\F}(x)$ (as in the proof of Theorem \ref{tch}). In doing that, we can prove the following theorem.

\begin{theorem}\label{tme}
For $a\not =0, 1$ and a pair $\F=(F_1,F_2)$ of finite sets of nonnegative integers the invarinace
\begin{equation}\label{iza}
\Mm_{\F,x}^{a,c}=(-1)^{w_\F}\Mm_{I(\F),-x}^{a,-c-\max F_1-\max F_2},
\end{equation}
holds (see (\ref{dinvn}) for the definition of $I(\F)$).
\end{theorem}

This theorem was conjecture in \cite{dume} (Conjecture 2 in the Introduction).

\section{Invariance for quasi Wroskian whose entries are Laguerre polynomials}\label{sla}
We write $(L_n^\alpha)_n$ for the sequence of Laguerre polynomials  (see \cite{Ch}, Ch. V; see also \cite{KLS}, pp, 241-244).

One can obtain Laguerre polynomials from Meixner polynomials using the limit
\begin{equation}\label{blmel}
\lim_{a\to 1}(a-1)^nm_n^{a,c}\left(\frac{x}{1-a}\right)=L_n^{c-1}(x)
\end{equation}
see \cite{KLS}, p. 243 (take into account that we are using for the
Meixner polynomials a different normalization to that in \cite{KLS})

Using the limit (\ref{blmel}), one can get from the quasi Casoratian (\ref{defmex})  the following quasi Wronskian whose entries are Laguerre polynomials
(see \cite{dume}, Section 5).
\begin{equation}\label{defhoml}
\Ll_{\F,x}^{\alpha}=(-1)^{\sum_{F_1}f}\left|
  \begin{array}{@{}c@{}lccc@{}c@{}}
   &  &&\hspace{-.9cm}{}_{1\le j\le k} \\
    \dosfilas{ (L_{f}^{\alpha})^{(j-1)}(x)}{f\in F_1} \\
    \dosfilas{ L_{f}^{\alpha +j-1}(-x) }{f\in F_2}
  \end{array}
  \right| ,
\end{equation}
where as before $\F=(F_1,F_2)$ is a pair of finite sets of nonnegative integers, $F_i$ with $k_i$ elements, $i=1,2$, respectively, and $k=k_1+k_2$.

More precisely
\begin{equation}\label{lim12}
\lim _{a\to 1^-}(1-a)^{w_\F}\Mm_{\F,x/(1-a)}^{a,c}=\Ll_{\F,x}^\alpha .
\end{equation}

Theorem \ref{tme} then gives.

\begin{theorem}\label{tla}
For a pair $\F=(F_1,F_2)$ of finite sets of nonnegative integers the invariance
\begin{equation}\label{ila}
L_{\F,x}^{\alpha}=(-1)^{w_\F}L_{I(\F),-x}^{-\alpha-\max F_1-\max F_2-2},
\end{equation}
holds.
\end{theorem}
This theorem was conjecture in \cite{dume} (see identity (1.12) in the Introduction).

\section{Invariance for quasi Casorati-Hahn determinants}\label{sha}
We write $(h_{n}^{\alpha,\beta,N})_n$ for the sequence of Hahn polynomials defined by
\begin{equation}\label{hpol}
h_{n}^{\alpha,\beta,N}(x)=\frac{(-N)_n(\alpha+1)_n}{n!}\pFq{3}{2}{-n,-x,n+\alpha+\beta+1}{\alpha+1,-N}{1}
\end{equation}
where as usual $\pFq{3}{2}{a,b,c}{d,e}{x}$ denotes the hypergeometric function
(we have taken a slightly different normalization from the one used in \cite{KLS}, pp, 234-7). Notice that when $\alpha+\beta\not =-1,-2,\cdots$, $h_{n}^{\alpha,\beta,N}$ is always a polynomial of degree $n$.

Hahn polynomials are not self-adjoint in the sense that they do not satisfy identities as (\ref{Chdua}) of (\ref{duame}) for the Charlier and Meixner polynomials, respectively. However they have a dual family: the dual Hahn polynomials.

For $\alpha \not =-1,-2,\cdots $ we write $(R_{n}^{\alpha,\beta,N})_n$ for the sequence of dual Hahn polynomials defined by
\begin{equation}\label{dhpol}
R_{n}^{\alpha,\beta,N}(x)=\frac{1}{n!}\sum _{j=0}^n\frac{(-n)_j(-N+j)_{n-j}(\alpha+1+j)_{n-j}}{(-1)^j j!}\prod_{i=0}^{j-1}(x-i(\alpha+\beta+1+i))
\end{equation}
(we have taken a slightly different normalization from the one used in \cite{KLS}, pp, 234-7). Notice that $R_{n}^{\alpha,\beta,N}$ is always a polynomial of degree $n$.
Using that
$$
(-1)^j\prod_{i=0}^{j-1}(\lambda^{\alpha,\beta} (x)-i(\alpha+\beta+1+i))=(-x)_j(x+\alpha+\beta+1)_j,
$$
where $\lambda^{\alpha,\beta}(x)=x(x+\alpha+\beta+1)$ (to simplify the notation we sometimes write $\lambda(x)=\lambda^{\alpha,\beta}(x)$), we get the hypergeometric representation
$$
R_{n}^{\alpha,\beta,N}(\lambda^{\alpha,\beta}(x))=\frac{(-N)_n(\alpha+1)_n}{n!}\pFq{3}{2}{-n,-x,x+\alpha+\beta+1}{\alpha+1,-N}{1}.
$$
The hypergeometric representation of dual Hahn and Hahn polynomials gives the duality: for $n,m\ge 0$
\begin{equation}\label{sdm2b}
\frac{(-N)_m(\alpha+1)_m}{m!} h_{n}^{\alpha,\beta,N}(m)=\frac{(-N)_n(\alpha+1)_n}{n!} R_{m}^{\alpha,\beta,N}(\lambda^{\alpha,\beta}(n)).
\end{equation}
When $N$ is not a nonnegative integer and $\alpha ,-\beta-N-1 \not =-1,-2,\cdots $, dual Hahn polynomials are always orthogonal with respect to a measure $\mu_{\alpha,\beta,N}$. When $N$ is a positive integer and $\alpha ,\beta \not =-1,-2,\cdots -N $, $\alpha+\beta \not=-1,\cdots, -2N-1$, we have
\begin{equation}\label{masdh}
\mu_{\alpha,\beta,N}=\sum _{x=0}^N \frac{(2x+\alpha+\beta+1)(\alpha+1)_x(-N)_xN!}{(-1)^x(x+\alpha+\beta+1)_{N+1}(\beta+1)_xx!}\delta_{\lambda(x)}.
\end{equation}
It turns out that $\langle R_n^{\alpha,\beta,N},R_n^{\alpha,\beta,N}\rangle\not =0$ only for $0\le n\le N$ (notice that the measure $\mu_{\alpha,\beta,N}$ is a finite combination of Dirac deltas).
The measure $\mu_{\alpha,\beta,N}$ is either positive or negative only when $N$ is a positive integer
and either $-1<\alpha,\beta$ or $\alpha,\beta<-N$, respectively.

Krall-dual Hahn measures seem to have a richer structure than the Krall-Charlier or Krall-Meixner measures considered in Section \ref{ich} and \ref{ime}. Indeed, given a trio $\F=(F_1,F_2,F_3)$ of finite sets of positive integers, $F_i$ with $k_i$ elements, $i=1,2,3$, respectively,
we write $k=k_1+k_2+k_3$. One can then construct
Krall-dual Hahn measures as follows (see \cite{dudh})
\begin{equation}\label{mii}
\mu_{\alpha,\beta,N}^{\F}=\prod_{f\in F_1}(\lambda-\lambda(f))\prod_{f\in F_2}(\lambda-\lambda(f-\beta))\prod_{f\in F_3}(\lambda-\lambda(N-f))\mu _{\alpha,\beta,N},
\end{equation}
Since this measure is a Christoffel transform of the measure $\mu _{\alpha,\beta,N}$ and it is also a Krall measure, we can proceed as in Section \ref{ich} and Section \ref{ime}. Due to the duality (\ref{sdm2b}) we get invariance properties for Casoratian determinants whose entries are Hahn polynomials. More precisely, consider the quasi Casorati-Hahn determinant
\begin{equation}\label{cdh0}
D_{\F,x}^{\alpha,\beta,N}=\frac{\left|
\begin{array}{@{}c@{}lccc@{}c@{}}
     &  &&\hspace{-.9cm}{}_{1\le j\le k} \\
     \dosfilas{ (\alpha+x+1,-N+x)_{j-1}h_{f}^{\alpha,\beta,N}(x+j-1) }{f\in F_1} \\
    \dosfilas{ (\alpha+x+1,-\beta-N+x)_{j-1}h_{f}^{\alpha,-\beta,\beta+N}(x+j-1) }{f\in F_2}\\
    \dosfilas{ (-N+x,-\beta-N+x)_{j-1}h_{f}^{-\beta-N-1,-\alpha-N-1,N}(x+j-1) }{f\in F_3}
  \end{array}
  \right|}{\displaystyle \prod_{s=1}^{3}\prod_{i=0}^{\tilde k_s-2}(\xi_s+x+i)^{\tilde k_s-i-1}},
\end{equation}
where $\tilde k_1=k_1+k_2$, $\tilde k_2=k_1+k_3$, $\tilde k_3=k_2+k_3$ and $\xi_1=\alpha+1$, $\xi_2=-N$, $\xi_3=-\beta-N$, and we use the notation $(a_1,\cdots, a_m)_j=(a_1)_j\cdots (a_m)_j$.

It can be proved (as Lemma 3.3 of \cite{dudh}) that this determinant is always a polynomial. Denote by $d_{\F}^{\alpha,\beta,N}$ its leading coefficient.
Under mild conditions on the parameters, the determinant (\ref{cdh0}) is a polynomial of degree $w_\F$ and
\begin{align*}
d_{\F}^{\alpha,\beta,N}=&(-1)^{k_1k_2+k_1k_3+k_2k_3}\prod_{j=1}^3V_{F_j}\prod_{f\in F_j}\frac{(f+\eta_j)_f}{f!}
\prod_{u\in F_1;v\in F_2}(\beta+u-v)\\&\prod_{u\in F_1;w\in F_3}(\alpha+\beta+N+1+u-w)
\prod_{v\in F_2;w\in F_3}(N+\alpha+1+u-w),
\end{align*}
where $\eta_1=\alpha+\beta+1$, $\eta_2=\alpha-\beta+1$ and $\eta_3=-\alpha-\beta-2N-1$.

Consider now the (normalized) quasi Casorati-Hahn determinant
\begin{equation}\label{cdh}
\Hh_{\F,x}^{\alpha,\beta,N}=\frac{D_{\F,x}^{\alpha,\beta,N}}{d_{\F}^{\alpha,\beta,N}},
\end{equation}
where when $d_{\F}^{\alpha,\beta,N}=0$ (and hence $D_{\F,x}^{\alpha,\beta,N}=0$), we set $\Hh_{\F,x}^{\alpha,\beta,N}=1$.

The following theorem can be proved as Theorem \ref{tch}.

\begin{theorem}\label{tha}
For a trio $\F=(F_1,F_2,F_3)$ of finite sets of nonnegative integers the invariance
\begin{equation}\label{iha}
H_{\F,x}^{\alpha,\beta,N}=\epsilon H_{I(\F),-x}^{-\alpha-\max F_1-\max F_2-2,-\beta-\max F_1+\max F_2,-N+\max F_1+\max F_3},
\end{equation}
holds, where $\epsilon $ is the sign $(-1)^r$, with $r$ the degree of $H_{\F,x}^{\alpha,\beta,N}$.
\end{theorem}

\section{Invariance for quasi Wroskian whose entries are Jacobi polynomials}\label{sja}
As usual, we write $(P_n^{\alpha,\beta})_n$ for the sequence of Jacobi polynomials (see \cite{Ch}, Ch. V; see also \cite{KLS}, pp, 250-3).
When $\alpha,\beta \not =-1,-2,\cdots$, the Jacobi polynomials are orthogonal with respect to a measure, which it is positive only when $\alpha,\beta>-1$ and then
it is equal to $(1-x)^\alpha(1+x)^\beta$, $x\in (-1,1)$.

One can obtain Jacobi polynomials from Hahn polynomials using the limit
\begin{equation}\label{blhj}
\lim_{N\to +\infty}\frac{h_n^{\alpha,\beta,N}\left(\frac{(1-x)N}{2}\right)}{(-N)_n}=P_n^{\alpha,\beta}(x)
\end{equation}
see \cite{KLS}, p. 207 (note that we are using for
Hahn polynomials a different normalization to that in \cite{KLS})

This limit together with Theorem \ref{tha} can be used to get invariance properties for quasi Wronskian whose entries are Jacobi polynomials. Since
for $y\in \RR$ \begin{equation}\label{blhje}
\lim_{N\to +\infty}\frac{h_n^{-\beta-N-1,-\alpha-N-1,N}\left(\frac{(1-x)N}{2}+y\right)}{(-N)_{2n}}=\frac{x^n}{n!},
\end{equation}
we only take limit in (\ref{cdh0}) when $F_3=\emptyset$.
In doing that (see \cite{duhj} for details), we get the quasi Wronskian
\begin{equation}\label{defhom0}
U_{\F,x}^{\alpha,\beta}=\frac{ \left|
  \begin{array}{@{}c@{}lccc@{}c@{}}
    &  &&\hspace{-.9cm}{}_{1\le j\le k} \\
    \dosfilas{ (-1)^{j-1}(P_{f}^{\alpha ,\beta })^{(j-1)}(x) }{f\in F_1} \\
    \dosfilas{ (\beta-f)_{j-1}(1+x)^{k-j}P_{f}^{\alpha +j-1 ,-\beta-j+1 }(x)}{f\in F_2}
  \end{array}
  \hspace{-.4cm}\right|}{(1+x)^{k_2(k_2-1)}},
\end{equation}
where as before $\F=(F_1,F_2)$ is a pair of finite sets of nonnegative integers, $F_i$ with $k_i$ elements, $i=1,2$, respectively, and $k=k_1+k_2$.
Since (\ref{cdh0}) is a polynomial, this quasi Wronskian is a polynomial as well. Let us write $u_\F^{\alpha,\beta}$ for its leading coefficient.

Under mild conditions on the parameters, the determinant (\ref{defhom0}) is a polynomial of degree $w_\F$ and
$$
u_\F^{\alpha,\beta}=\frac{V_{F_1}V_{F_2}\prod_{f\in F_1,F_2}(\alpha+\epsilon_f\beta+f+1)_f\prod_{u\in F_1,v\in F_2}(\beta +u-v)}{(-1)^{\binom{k_1}{2}+\binom{k_2}{2}}2^{\sum_{f\in F_1,F_2}f}\prod_{f\in F_1,F_2}f!},
$$
where $\epsilon_f=1$ for $f\in F_1$ and $\epsilon_f=-1$ for $f\in F_2$.

Consider now the (normalized) quasi Wronskian
\begin{equation}\label{defhom}
\Pp_{\F,x}^{\alpha,\beta}=\frac{U_{\F,x}^{\alpha,\beta} }{u_{\F}^{\alpha,\beta}}.
\end{equation}

Theorem \ref{tha} then gives.

\begin{theorem}\label{tja}
For a pair $\F=(F_1,F_2)$ of finite sets of nonnegative integers the invariance
\begin{equation}\label{ija}
P_{\F,x}^{\alpha,\beta}=P_{I(\F),x}^{-\alpha-\max F1-\max F_2-2,-\beta-\max F1+\max F_2}
\end{equation}
holds.
\end{theorem}

\noindent
\textit{Mathematics Subject Classification: 42C05, 33C45, 33E30}

\noindent
\textit{Key words and phrases}: Orthogonal polynomials. Exceptional orthogonal polynomial. Krall orthogonal polynomials.
Charlier polynomials. Hermite polynomials. Meixner polynomials. Laguerre polynomials. Hahn polynomials. Jacobi polynomials.

     \end{document}